\documentclass{amsproc}

\usepackage[backref]{hyperref}
\usepackage[alphabetic,backrefs,lite]{amsrefs}
\usepackage[usenames,dvipsnames]{color}

\newtheorem{tm}{Theorem}[section]
\newtheorem{pr}[tm]{Proposition}

\newtheorem{lm}[tm]{Lemma}

\newtheorem{rmk}[tm]{Remark}

\newcommand{\Aone}{{\mathbb{A}^{\!1}}}

\newcommand{\GW}{\mathrm{GW}}

\newcommand{\proj}{ \operatorname{Proj}}

\newcommand{\Hom}{\operatorname{Hom}} 

\newcommand{\Pic}{\operatorname{Pic}}

\newcommand{\ind}{\operatorname{ind}}

\newcommand{\hidden}[1]{\footnote{Hidden:  #1}}
\renewcommand{\hidden}[1]{}

\newcommand{\bbA}{\mathbf{A}}

\newcommand{\bbF}{\mathbf{F}}

\newcommand{\bbP}{\mathbf{P}}

\newcommand{\bbR}{\mathbf{R}}

\newcommand{\calO}{{ \mathcal O}}

\newcommand{\Spec}{\operatorname{Spec}}

\newcommand{\ShHom}{\mathscr{H}\kern -.5pt om}

\numberwithin{equation}{section}

\begin{document}

\title[Enriched wild ramification]{Examples of wild ramification in an enriched Riemann--Hurwitz formula}


\author{Candace Bethea}

\author{Jesse Leo Kass}

\author{Kirsten Wickelgren}

\subjclass[2010]{Primary 14F42}

\date{}

\begin{abstract}
M. Levine proved an enrichment of the classical Riemann--Hurwitz formula to an equality in the Grothendieck--Witt group of quadratic forms. In its strongest form, Levine's theorem includes a technical hypothesis on ramification relevant in positive characteristic. We describe  what happens when the hypotheses are weakened by  showing an analogous Riemann--Hurwitz formula and describing an example suggested by S. Saito. 
\end{abstract}

\maketitle

\section{Introduction}

In the recent preprint \cite{Levine-EC}, Marc Levine established an enriched version of the Riemann--Hurwitz formula that is valued in the Grothendieck--Witt group of nondegenerate symmetric bilinear forms.  In positive characteristic, especially over an imperfect field, the strongest form of Levine's theorem includes technical hypotheses on the ramification. At the workshop {\em Motivic homotopy theory and refined enumerative geometry}, Shuji Saito asked whether a Riemann--Hurwitz formula should hold more generally.  As an illustration of the situation he was interested in, he gave the example of the map $\bbP^{1}_{\bbF_{p}(t)} \to \bbP^{1}_{\bbF_{p}(t)}$ defined by $y\mapsto\frac{t-y^p}{y}$.  In this article, we prove a theorem which strengthens Levine's result by establishing his result under weaker hypotheses so that it applies to Saito's function.  We also illustrate the content of the theorem by computing the terms for Saito's example as well as for  the Artin--Schreier cover $\bbP^{1}_{\bbF_{p}(t)} \to \bbP^{1}_{\bbF_{p}(t)}$ defined by $y\mapsto y^{p}-y$.

The problem of establishing an enriched Riemann--Hurwitz formula with weaker hypotheses in characteristic $p$ is interesting because  the classical Riemann--Hurwitz formula becomes more complicated when passing from characteristic $0$ to characteristic $p>0$.   Over an algebraically closed field of characteristic $0$, the formula, as described in Corollary~2.4 of \cite[Chapter 4]{hartshorne}, states that a nonconstant map $f \colon Y \to X$ of curves satisfies
\begin{equation} \label{eq:UnenrichedRHCharZero}
	 \chi(Y) = d\cdot\chi(X)- \sum (e(y)-1).
\end{equation}
Here $d$ is the degree of $f$, $\chi(C)=2-2 g(C)$ is the topological Euler characteristic of a curve $C$, and $e(y)$ is the ramification index of $f$.  Recall the ramification index  is the normalized valuation $\nu_{y}(f^{*}(t))$ of the pullback of a uniformizer $t \in \calO_{X, f(y)}$.

When $k$ has positive characteristic (but is still assumed to be algebraically closed), Formula~\eqref{eq:UnenrichedRHCharZero} becomes more complicated in two ways.  First, we need to additionally require that $f$ is separable (to avoid maps like $y \mapsto y^p$ which is ramified everywhere).  Second, Equation~\eqref{eq:UnenrichedRHCharZero} holds as stated when $f$ is separable and the ramification indices are all coprime to $p$, i.e.~when $f$ is tamely ramified, but in general, the term $e(y)-1$ must be modified.  Define the branch index $b(y)$ to be the length of the module of relative K\"{a}hler differentials, i.e.~$b(y) := \operatorname{length}(\Omega_{Y/X, y})$. We then have 
\begin{equation} \label{Eqn: UnenrichedRHAllChar}
	 \chi(Y) = d\cdot\chi(X)- \sum b(y),
\end{equation}
and $b(y) \ge e(y)-1$ with equality holding if and only if $e(y)$ is coprime to $p$.  The branch index can alternatively be described in terms of the uniformizers.  If $t \in \calO_{X, f(y)}$ and $u \in \calO_{Y, y}$ are uniformizers, then the branch index equals the valuation $v_{y}(dt/du)$ for $dt/du$ the unique function satisfying $f^{*}( dt ) = dt/du \cdot du$.  

Formula~\eqref{Eqn: UnenrichedRHAllChar} remains valid when $k$ is nonalgebraically closed  provided the branch index is defined by $b(y) := \operatorname{length}(\Omega_{Y/X, y})$ \cite[Theorem~4.16, Remark~4.17]{qing}.  The branch index does not, however, always equal $v_{y}(dt/du)$.  Indeed, for Saito's example $y \mapsto \frac{t-y^p}{y}$, we have $dt/du=0$.  In this example, the residual extension $k(y)/k(x)$ is inseparable.  When the residual extension is separable, an explicit expression for $b(y)$ is given by  \cite[Chapter~3, Propositions~13 and 14; Chapter~4, Proposition~4]{serre79} (where $b(y)$ appears as the valuation of the different).

Over the real numbers $k=\bbR$, the Riemann--Hurwitz formula admits a real-topological analogue.  One implication of the hypotheses is that the manifold of real points $Y(\bbR)$ is orientable. Once we fix an orientation, the topological degree $\deg^{\bbR}(f)$ of the map $f \colon Y(\bbR) \to X(\bbR)$ on real points is well-defined and satisfies an analogue of the Riemann--Hurwitz formula, as was observed by Levine in  \cite[Example~12.9]{Levine-EC}.  Specifically, if $y \in Y$ is a point with residue field $\bbR$ and $t \in \calO_{X, f(y)}$, $u \in \calO_{Y, y}$ are uniformizers with $t$ compatible with the orientation (so the function germ $t \colon X(\bbR) \to \bbR$ is orientation-preserving at $f(y)$), then define the real branch index $b^{\bbR}(y)$ to be the local degree of $dt/du$ at $y$, so 
\[
	b^{\bbR}(y) = \begin{cases}
					+1	&	\text{ if $dt/du \circ u^{-1}$ is increasing at $0$;} \\
					-1	&	\text{ if $dt/du \circ u^{-1}$ is decreasing at $0$;} \\
					0	&	\text{ otherwise.}
				\end{cases}
\]
With this definition, we have
\begin{equation} \label{Eqn: RealRH}
	 \chi^{\bbR}(X) = d\cdot\chi^{\bbR}(Y)- \sum b^{\bbR}(y).
\end{equation}
Here $\chi^{\bbR}(X)$ denotes the Euler characteristic of $X(\bbR)$ of the real locus.  These Euler characteristics vanish, so the formula simplifies to 
\[
  	0 =  \sum b^{\bbR}(y).
\]
	
This equation admits a particularly simple interpretation when $X=Y=\bbP^{1}_{\bbR}$ and $f$ is defined by a monic polynomial $f = x^d + a_1 x^{d-1}+a_2 x^{d-2}+\dots+a_{d}$.  Considering $f$ as a continuous function $f \colon \bbR \to \bbR$, a computation of $b(\infty)$ shows that Equation~\eqref{Eqn: RealRH} takes the form
\[
	\#\text{local maxima of $f$}-\#\text{local minima of $f$}  = \begin{cases}
																					0	&	\text{ if $d$ is odd;} \\
																					-1	&	\text{ if $d$ is even.}
																				\end{cases}
\]

This result is the real realization of Levine's enriched Riemann--Hurwitz formula.  Over an arbitrary field, we replace the choice of an orientation of $X(\bbR)$ with the choices of a line bundle $M$ and an isomorphism $\alpha \colon M^{\otimes 2} \cong \operatorname{T}(X)$  of the square of $M$ with the tangent bundle.  Observe that, over $\bbR$, the pair $(M, \alpha)$ determines an orientation of $X(\bbR)$, but not every real curve admits a pair $(M, \alpha)$.  (Consider, for example, the Brauer--Severi curve $\{ X^2+Y^2+Z^2=0 \} \subset \bbP^{2}_{\bbR}$.)

Given $f\colon Y\to X$, Levine's strongest form of the enriched Riemann--Hurwitz formula holds under technical assumptions on the ramification that we recall below.  Let $t$ and $u$ be uniformizers as before, but now require that $t$ is compatible with $(M, \alpha)$ in the sense that, under the isomorphism on stalks $\alpha_{f(y)}^{\vee} \colon M^{- \otimes 2}_{f(y)} \cong \operatorname{T}^{\vee}_{x}(X)$, $dt$ corresponds to a tensor of the form $s \otimes s$ (rather than $s \otimes s'$ for $s \ne s'$).  

 Write the pullback $f^{*}(t)$ of the uniformizing parameter $t$ under the natural homomorphism $f^{*} \colon \calO_{X, f(x)} \to \calO_{Y, y}$ as  $f^{*}(t) = a\cdot u^{e(y)}$ with  $a\in \calO_{Y, y}^{*}$.  (Such an expression exists by the definition of $e(y)$.)  Define the motivic branch index by
\[
	b^{\bbA^{1}}(y) = \langle a(y) e(y) \rangle \cdot \sum_{i=0}^{e-2} \langle (-1)^{i} \rangle \text{ in  $\operatorname{GW}(k(y))$.}
\]
Here $\operatorname{GW}(k(y))$ denotes the Grothendieck--Witt group of nondegenerate symmetric bilinear forms, $\langle u \rangle$  denotes the class of the rank 1 bilinear form with Gram matrix $\begin{pmatrix} u \end{pmatrix}$, and $a(y)$ denotes the image of $a \in \calO_{Y, y}^*$ in the residue field.

With this notation, the hypothesis to \cite[Theorem~12.7]{Levine-EC} is that  $\operatorname{char} k \ne 2$,  $f$ is separable, and every ramification point has the property that $k(y)$ is a separable extension of $k$ and $e(y)$ is coprime to $p$.  When these hypotheses are satisfied, the theorem states
\begin{equation} \label{Eqn: LevineRiemannHurwitz}
		\chi^{\bbA^{1}}(Y)		=	d\cdot \chi^{\bbA^{1}}(X)  - \sum_{\{y\in Y\colon e(y)>1\}} \operatorname{Tr}_{k(y)/k}(b^{\bbA^{1}}(y))
\end{equation}
Here the sum runs over all ramification points of $f$.  The term $\chi^{\bbA^{1}}(C)$  is the Euler characteristic in $\bbA^1$-homotopy theory which equals $(1-g(C)) \cdot (\langle +1\rangle + \langle -1 \rangle)$  when $C$ is a smooth curve.  The expression $\operatorname{Tr}_{k(y)/k}(b^{\bbA^{1}}(y))$ is the class of the composition $\operatorname{Tr}_{k(y)/k} \circ \beta$ of the field trace map $\operatorname{Tr}_{k(y)/k} \colon k(y) \to k$ with a bilinear form $\beta$ representing $b^{\bbA^1}(y)$.

Equation~\eqref{Eqn: LevineRiemannHurwitz} is an enrichment of the earlier Riemann--Hurwitz formulas in the sense that those earlier statements can be deduced from it by comparing invariants. The fact that the left-hand  and right-hand sides of \eqref{Eqn: LevineRiemannHurwitz} have the same rank  is Equation~ \eqref{Eqn: UnenrichedRHAllChar}.  When $k=\bbR$, the fact that the signatures of the sides are equal is Equation~\eqref{Eqn: RealRH}.

Levine's hypotheses on the ramification points fail to hold  in Saito's example and in the example of the Artin--Schreier cover.  The Artin--Schreier cover fails to satisfy the hypotheses because $b(\infty)=p$.  In Saito's example, the ramification point $y$ defined by the ideal $(y^p-t)$ has the property that $e(y)=1$ but the residue extension $k(y)/k(f(y))$ is inseparable.  Levine deduces \cite[Theorem~12.7]{Levine-EC} from \cite[Corollary~10.9]{Levine-EC}, 
and that corollary applies when $f$ is wildly ramified, but it does not provide an explicit expression for $e^{\bbA^1}(y)$.  As Levine remarks immediately after the corollary, the main result of \cite{KWA1degree} can be used to derive an explicit expression for these branch indices.  In fact, we can reduce to the case where $\calO_{Y, y}$ is a monogenic extension of $\calO_{X, f(y)}$ by Proposition~\ref{NisnevichCoord} below, and in this case the branch indices can be computed using the earlier work of Cazanave \cite{cazanavea, cazanaveb} (loc.~cit.~only treat the global $\bbA^1$-degree, but see \cite{KWA2degree} for the relation with the local $\bbA^1$-degree).

In this paper, we explain in more detail how to use the results of \cite{KWA1degree} to establish an enriched Riemann--Hurwitz formula with explicitly computable branch indices when $f$ is allowed to have wild ramification. Rather than using the formalism developed in \cite{Levine-EC}, we establish an enriched Riemann--Hurwitz formula using the Euler class formalism in \cite{CubicSurface}.  Under suitable hypotheses, the local index $\ind_y df$ is defined as the local $\bbA^1$-degree with respect to a coordinate system.  This class is represented by an explicit bilinear form, and we recall a recipe for computing the form in Section~\ref{Notationsection}. 

The main result is

\begin{tm}\label{IntroRHthm}
Let $k$ be any field. Let $f: Y \to X$ be a non-constant, separable map of smooth, proper, geometrically connected curves over $k$. We make the following assumptions:
\begin{enumerate}
\item \label{assumeTXsquare} $X$ is oriented, meaning we have a line bundle $M$ and a chosen isomorphism  $M^{\otimes 2} \cong T^{\vee}X$.
\item \label{technical} $e(Y, \Hom(f^* T^{\vee}X ,T^{\vee}Y), df) = e(Y, \Hom(f^* T^{\vee}X ,T^{\vee}Y), \alpha df)$ for all $\alpha$ in $k^*$.
\end{enumerate}

Then there is an equality
\begin{equation} \label{Eqn: FinalRH}
		\chi^{\bbA^{1}}(Y)		=	d\cdot\chi^{\bbA^{1}}(X)  - \sum_{\{y: df(y) = 0\}} \ind_y df.
\end{equation}
\end{tm}

This is proven in Section~\ref{Section:MT} below. The notation $e(Y, \Hom(f^* T^{\vee}X ,T^{\vee}Y), df)$ is defined in Section~\ref{Notationsection}. As we explain, when $f \colon Y \to X$ is described explicitly, the local indices of $df$ can be effectively computed using the main results of \cite{cazanavea, cazanaveb, KWA1degree}. 

\begin{rmk} We comment on the assumptions in Theorem~\ref{IntroRHthm}.
\begin{enumerate}
\item Assumption 1 can always be achieved after base change to a finite extension $M$ of $k$, because $T^{\vee} X$ has even degree and $\Pic^0(X)(\bar k)$ is divisible. 
\item There is work in progress to show that Assumption \eqref{technical} always holds \cite{Bachmann}, as well as to show that the local indices agree with those of \cite{Levine-EC}.
\end{enumerate}
\end{rmk}

We demonstrate the theorem in Section~\ref{Section:Example} by explicitly working out the terms in Equation~\eqref{Eqn: FinalRH} for the Artin--Schreier cover and Saito's function $f(y) = (t-y^{p})/y$.   These examples are especially interesting in the context of work of Kato and Abbes--Saito \cite{kato89, abbes02, abbes03, saito12} on the conductor of an extension of local fields with inseparable residue extension because classical work suggests that it would be interesting to explore the connection between that conductor and the term $b^{\bbA^{1}}(y)$.  Recall that, when the residue field is separable, the classical index $b(y)$  equals the valuation of the different which is related to the conductor by Artin's Conductor-Discriminant formula \cite[Chapter~VII,(11.9)]{neukirch}. Perhaps there is an enrichment of this formula that extends to the case where the residual extension is inseparable.

\section{Notation}\label{Notationsection}

For a field $k$, let $\GW(k)$ denote the Grothendieck--Witt group of $k$, which is the group completion of the semi-ring under $\oplus$ and $\otimes$ of isomorphism classes of $k$-valued nondegenerate symmetric bilinear forms on finite dimensional $k$-vector spaces. Since all such forms are stably diagonalizable, $\GW(k)$ is generated by $1$-dimensional forms $\langle a \rangle$ with $a$ in $k^*/(k^*)^2$, where $\langle a \rangle$ is the isomorphism class generated by the bilinear form $$ k \times k \to k, $$ $$(x,y) \mapsto a xy.$$ The class of the hyperbolic form is denoted $h$ and is given by $h = \langle 1 \rangle + \langle -1 \rangle.$

We recall some definitions from \cite{CubicSurface} that allow us to define an Euler number in $\GW(k)$. There are other definitions of such Euler numbers. Relevant references include \cite{Grig_Ivan}, \cite{BargeMorel},  \cite{FaselGroupesCW}, \cite[Chapter 8.2]{morel}, \cite{AsokFasel_comp_euler_classes},  \cite{DJK}, and \cite{LevineRaksit_MotivicGaussBonnet}. Please see, for example, the discussion in Section 1.1 of \cite{CubicSurface} entitled {\em Relation to other work}.

Let $Y$ be a smooth $k$-scheme of dimension $r$. Given a point $y$ of $Y$, {\em Nisnevich coordinates around $y$} are the data of a Zariski open neighborhood $U$ of $y$ and an \'etale map $\phi: U \to \mathbb{A}^r_k$ from $U$ to affine $r$-space such that the induced map $k(\phi(y)) \to k(y)$ on residue fields is an isomorphism. 

Let $\mathcal{V} \to Y$ be a vector bundle of rank $r$ over $Y$. A relative orientation of $\mathcal{V}$ is a line bundle $M$ and an isomorphism $\Hom(\det TY, \det \mathcal{V})\cong M^{\otimes 2}$. 
 
Let $\sigma$ be a section of $\mathcal{V}$. Given Nisnevich local coordinates around an isolated zero $y$ of $\sigma$, there is a local index (also called a {\em local degree}) $\ind_y \sigma$ in $\GW(k)$ of $\sigma$ at $y$ defined in \cite[Definition 28]{CubicSurface}.  Because our interest lies in being able to compute the local indices at points, especially points whose residue fields are inseparable extensions of $k$ and when the order of vanishing of $\sigma$ is divisible by the characteristic $k$, we recall the following computational recipe for $\ind_y \sigma$. One can also compute $\ind_y \sigma$ using the main result of \cite{cazanaveb}.

We choose a local trivialization of $\mathcal{V}$ near $y$ which is compatible with the relative orientation. Under this trivialization, $\sigma$ is identified with an element $f$ of $\oplus_{i=1}^r \mathcal{O}_Y$. Let $m_y$ denote the ideal corresponding to $y$. We can choose an element $g$ of $\oplus_{i = 1}^r m_y^N$ for $N$ sufficiently large (relative to the order of vanishing of $f$) so that the function $f+g$ is in the image of $\phi^*: \oplus_{i=1}^r \mathcal{O}_{\mathbb{A}^r_k, \phi(y)} \to \oplus_{i=1}^r \mathcal{O}_{Y,y}$. Choose $F$ in $\oplus_{i=1}^r \mathcal{O}_{\mathbb{A}^r_k,\phi(y)}$ such that $\phi^* (F) = f+g$. (Any choice of such a $g$ and $F$ will do.) Then $F$ determines a function $F: W \to \mathbb{A}^r_k $ from an open subset $W \subset \mathbb{A}^r_k = \Spec k[y_1,\ldots,y_r]$ to  $\mathbb{A}^r_k$. Let $F=(F_1,F_2,\ldots,F_r)$. 

Since $y$ is an isolated zero, $Q=k[y_1,\ldots,y_r]_{\phi(y)}/\langle F_1, \ldots, F_r \rangle$ is a finite dimensional $k$-vector space. Scheja--Storch \cite[Section~3]{scheja} construct the following bilinear form on $Q$, and the isomorphism class of this form is $\ind_y \sigma$. We can choose $a_{ij}$ in $k[y_1,\ldots,y_r] \otimes_k k[y_1,\ldots,y_r] $ such that $$F_j \otimes 1 - 1 \otimes F_j = \sum_i a_{ij} (y_i \otimes 1 - 1 \otimes y_i).$$ Let $\Delta$ denote the image of $\det (a_{ij})$ in $Q \otimes_k Q$. There is a canonical map $Q \otimes_k Q \to \Hom(\Hom(Q,k),Q)$ sending $b \otimes c$ to the linear map which sends $\mu$ to $\mu(b)c$. The image of $\Delta$ is an isomorphism $\Theta: \Hom(Q,k) \to Q$. Let $\eta = \Theta^{-1}(1)$. We obtain a bilinear form $Q \times Q \to k$ defined by $$(b,c) \mapsto \eta(bc).$$

It follows from the main theorem of \cite{KWA1degree} that $\ind_y \sigma$ agrees with the local $\mathbb{A}^1$-degree of the associated function  $F$ at $\phi(y)$, at least when $y$ is $k$-rational or $F$ has a simple zero. We denote this latter element of $\GW(k)$ by $\deg^{\mathbb{A}^1}_y \sigma$. 

Suppose now that additionally $Y$ is  proper, $\mathcal{V}$ is relatively oriented, and that $\sigma$ is a section with only isolated zeros such that there are Nisnevich coordinates around every zero of $\sigma$. Then define the Euler number $e(Y, \mathcal{V}, \sigma)$ of $\mathcal{V}$ with respect to $\sigma$  by $$e(Y, \mathcal{V}, \sigma) = \sum_{x: \sigma(x) = 0} \ind_x \sigma.$$  If $\sigma$ and $\sigma'$ are sections with only isolated zeros that can be connected by sections with only isolated zeros, potentially after base change by an odd degree field extension, then $e(Y, \mathcal{V}, \sigma) = e(Y, \mathcal{V}, \sigma')$ (\cite[Corollary 36]{CubicSurface}). 

\section{Examples}\label{Section:Example}
Here we look at the Riemann--Hurwitz formula in two specific cases that illustrate some of the more delicate behavior of covers of curves in characteristic $p$.  The first example is the example suggested by Shuji Saito.  This example is a rational map $f \colon \bbP^{1}_{k(t)} \to \bbP^{1}_{k(t)}$ with a ramification point $y$ such that $e(y)=1$ but the residual extension $k(y)/k(f(y))$ is inseparable. 

The second example is an Artin--Schreier cover $f \colon \bbP^{1}_{k} \to \bbP^{1}_{k}$.  This cover has the property that it admits a ramification point $y$ such that $k(y)=k(f(y))$ but the ramification index $e(y)$ equals the residual characteristic $p$.  


We work over a field $k$ of odd characteristic $p$.  For a field extension $L/k$, let  $ \bbP^1_{L, y}$ denote $1$-dimensional projective space over the field $L$ with projective variables $Y$ and $W$, i.e., $\bbP^1_{L,y} = \proj L[W,Y]$ with affine coordinate $y=Y/W$, and let $\bbP^1_{L ,x}=\proj L[X,Z]$.  

\begin{pr}\label{Example} The rational function  $f: \bbP^1_{k(t),y} \to \bbP^1_{k(t),x}$ defined by $y\mapsto\frac{t-y^p}{y}$ has the property that there is a relative orientation such that the branch indices are
\begin{align*}
	\ind_{(y^p-t)}df=& \left(\frac{p-1}{2}\right)\cdot h+\langle 1\rangle \\
	\ind_{\infty}df=&  \left(\frac{p-3}{2}\right)\cdot h+\langle -1\rangle.
\end{align*}

In particular, the enriched Riemann--Hurwitz formula \eqref{Eqn: FinalRH} holds.
 \end{pr}

\begin{proof} We will show there is an equality $$\sum_{\{y: df(y) = 0\}} \ind_y ^\Aone df = h\left(g(\bbP^1_{k(t),y} )-1 + \deg f(1-g
(\bbP^1_{k(t),x}))\right)$$ in GW$(k(t))$, where $\deg f$ refers to the degree of the extension of function fields. First observe that the right-hand side is 
\begin{align*}h\left(g(\bbP^1_{k(t),y} )-1 + \deg f(1-g
(\bbP^1_{k(t),x}))\right) &=h(0-1+p(1-0))\\ &=h(p-1),\end{align*} so we need to show that $\sum_{\{y\colon df(y)=0\}}\ind_y^\Aone df=h(p-1)$. Observe that $$\{y\colon df(y)=0\}=\{y\colon f'(y)=0  \text{ or } h'(w)=0\},$$ where $w=\frac{1}{y}$ and $h=\frac{1}{f}$. Here we are using the affine coordinates $y=Y/W$ on $\Spec k(t)[y]$ and $w=W/Y$ on $\Spec k(t)[w]$. On $\Spec k(t)[y,y^{-1}]$, $x=f(y)=\frac{t-y^p}{y}$ and $f'(y)=\frac{y^p-t}{y^2}=0$ only at the point defined by the ideal $(y^p-t)$. On $\Spec k(t)[w]$, $z=h(w)=\frac{w^{p-1}}{w^pt-1}$ and $h'(w)=\frac{(p-1)w^{p-2}}{w^pt-1}=0$ at $w=\frac{1}{y}=0$. We will compute $\ind_{(y^p-t)}^\Aone df$ and $\ind_{\infty}^\Aone df$ separately. 

In order for local degrees to be well-defined, we must first construct a relative orientation for the line bundle $\Hom(f^*T^{\vee}\bbP^1_{k(t),x},T^{\vee}\bbP^1_{k(t),y})\to \bbP^1_{k(t),y}$. We will use the definition of relative orientability given by Kass and Wickelgren in \cite{CubicSurface}, which is related to the analogous definition in \cite{okonek14}. Given a line bundle $E$ on a smooth curve $C$, a relative orientation of $E$ is the datum of a line bundle $M$ and an isomorphism $\Hom(TC, E)\cong M^{\otimes 2}$. Observe that 
\begin{align*}
\Hom(T\bbP^1_{k(t),y}, \Hom(f^*T^{\vee}\bbP^1_{k(t),x},T^{\vee}\bbP^1_{k(t),y}) ) &\cong T^{\vee}\bbP^1_{k(t),y}\otimes \Hom(f^*T^{\vee}\bbP^1_{k(t),x},T^{\vee}\bbP^1_{k(t),y}) \\
&\cong T^{\vee}\bbP^1_{k(t),y}\otimes(f^{*}T^{\vee}\bbP^1_{k(t),x})^{\vee}\otimes T^{\vee}\bbP^1_{k(t),y}\\
&\cong \calO(-2)\otimes\calO(2p)\otimes\calO(-2)\\
&\cong \calO(p-2)^{\otimes 2},\\
\end{align*}
so $\Hom(f^*T^{\vee}\bbP^1_{k(t),x},T^{\vee}\bbP^1_{k(t),y})$ is relatively orientable.

We will still make precise the explicit isomorphism $T\bbP^1_{k(t),y}\cong \calO(1)^{\otimes 2}$ in order to compute local degrees using sections. Write $v:=-w$ on $\Spec (k(t)[v])=\Spec (k(t)[-w])$ and give $\bbP^1_{k(t),y}$ the affine coordinate $y=\frac{1}{v}$ so that $dy=\frac{1}{v^2}dv$. Define $T\bbP^1_{k(t),y}\to \calO(1)^{\otimes 2}$ on affine patches by $v\mapsto \partial y$ on $\Spec k(t)[v]$ and $y\mapsto \partial v$ on $\Spec k(t)[y]$, where $\partial y$ and $\partial v$ are the respective duals of the sections $dy$ and $dv$ on $T^{\vee}\bbP^1_{k(t),y}$. Note then that $\partial y=v^2\partial v$, as desired.

Let $U=f^{-1}\Spec(k(t)[z])\cap \Spec(k(t)[v])$ for notational ease, and let $\psi$ be the trivialization of $\Hom(f^*T^{\vee}\bbP^1_{k(t),x}, T^{\vee}\bbP^1_{k(t),y})$ which associates $1\in \calO_{\bbP^1_{k(t),y}}$ to the function $\{dz\mapsto dv\}$. Since $z=h(w)=\frac{w^{p-1}}{w^pt-1}$, $dz=\frac{(1-p)v^{p-2}}{v^pt-1}dv$ and hence $df|_{U}=\frac{(1-p)v^{p-2}}{v^pt-1}\{dz\mapsto dv\}$ in $\Hom(f^*T^{\vee}\bbP^1_{k(t),x},T^{\vee}\bbP^1_{k(t),y})(U)$. Thus the global section $df$ of $\Hom(f^*T^{\vee}\bbP^1_{k(t),x},T^{\vee}\bbP^1_{k(t),y})$ corresponds to the function $v\mapsto \frac{(1-p)v^{p-2}}{-v^pt-1}$ in the coordinate $v$ on $U$. Therefore we need to compute $\ind_{(y^p-t)}^\Aone df=\deg_{(y^p-t)}\frac{y^p-t}{y^2}$ and  $\ind_{\infty}^\Aone df=\deg_0^\Aone \frac{(1-p)v^{p-2}}{-v^pt-1}$.

The local degrees $\deg_{(y^p-t)}\frac{y^p-t}{y^2}$ and $\deg_0^\Aone \frac{(1-p)v^{p-2}}{-v^pt-1}$ can be computed using Cazanave's result on the naive homotopy class of a rational function from $\bbP^1_L$ to itself for any field $L$ \cite{cazanavea}. Given a  rational function $\frac{f_1}{f_2}\colon \bbP^1_L\to \bbP^1_L$, we can write $\frac{f_1(x)f_2(y)-f_1(y)f_2(x)}{x-y}=:\sum_{1\leq i,j\leq n}c_{ij}x^{i-1}y^{j-1}$. The B\'{e}zoutian of $\frac{f_1}{f_2}$, denoted B\'{e}z($f_1, f_2$), is defined to be the bilinear form with Gram matrix $[c_{ij}]_{1\leq i,j\leq n}$. Cazanave's main result is that B\'{e}z($f_1, f_2$) is a representative of the isomorphism class of $\deg^{\Aone}\frac{f_1}{f_2}$ in GW(L).

Cazanave's result allows us to compute the global degree, $\deg^\Aone(\frac{y^p-t}{y^2})$, which is equal to $\sum_{\{q\colon q\mapsto 0\}}\deg_q^\Aone\frac{y^p-t}{y^2}$ by \cite{KWA1degree}. In this particular case, the only zero of $f$ is the closed point $\{(y^p-t)\}$, so a global degree computation using the B\'{e}zoutian also computes the local degree, $\deg_{(y^p-t)}^\Aone \frac{y^p-t}{y^2}$. If we write $f_1=y^p-t$ and $f_2=y^2$, then 
\begin{align*}
\frac{f_1(x)f_2(y)-f_1(y)f_2(x)}{x-y}&=\frac{x^2y^2(x^{p-2}-y^{p-2})+t(x^2-y^2)}{x-y}\\
&=x^2y^2(x^{p-3}+x^{p-4}y+\cdots +y^{p-3})+t(x+y)\\
&=t(x+y)+x^{p-1}y^2+x^{p-2}y^3+\cdots +x^2y^{p-1}.\\
\end{align*}
Thus the Gram matrix of the B\'{e}zoutian of $\frac{y^p-t}{y^2}$, and hence a Gram matrix of $\ind_{(y^p-t)}^\Aone df$, is 
\[
\text{B\'{e}z}(y^p-t,y^2)= 
\begin{bmatrix}
    0       & t & 0 & \dots & 0 \\
    t       & 0 & 0 & \dots &0 \\
    0 & 0 & 0 & \dots & 1\\
    \vdots & \vdots &\vdots & \cdots & \vdots \\
   0     & 0 & 1 & \dots & 0
\end{bmatrix}.
\]

A diagonalization of this matrix is the Gram matrix of the diagonal nondegenerate symmetric bilinear form $\left(\frac{p-1}{2}\right)\cdot h+\langle 1\rangle$. Therefore we can conclude that $\ind_{(y^p-t)}^\Aone df=\left(\frac{p-1}{2}\right)\cdot h+\langle 1\rangle$ in GW($k(t)$).

Now we will compute $\deg_0^\Aone \frac{(1-p)v^{p-2}}{-v^pt-1}=\ind_{\infty}^\Aone df$. The global degree of $\frac{(1-p)v^{p-2}}{-v^pt-1}$ will be equal to the local degree at 0 plus the local degree at $\infty$, as $\{0,\infty\}$ is the fiber over $0$.  If we write $f_1=(1-p)v^{p-2}$ and $f_2=-v^pt-1$, then 
\begin{align*}
\frac{f_1(x)f_2(y)-f_1(y)f_2(x)}{x-y}&=\frac{(p-1)tx^{p-2}y^{p-2}(y^2-x^2)+(p-1)(y^{p-2}-x^{p-2})}{x-y} \\
&=(p-1)[tx^{p-2}y^{p-2}(y+x)+(y^{p-3}+y^{p-4}x+\cdots+yx^{p-4}+x^{p-3})]\\
&=(p-1)[tx^{p-2}y^{p-1}+tx^{p-1}y^{p-2}+y^{p-3}+y^{p-4}x+\cdots+yx^{p-4}+x^{p-3}].
\end{align*}
Thus the Gram matrix of the B\'{e}zoutian of $\frac{(1-p)v^{p-2}}{-v^pt-1}$ is 
\[
 \begin{bmatrix}
    0  & \dots &0& (p-1) & 0 & 0 \\
    0 & \dots &(p-1)& 0 &0 & 0 \\
    0 & \dots &0 & 0 & 0 & 0\\
    \vdots &\ddots & 0& 0 & \vdots &\vdots \\
    (p-1) & \dots & 0 & 0 & 0 & 0\\
   0 & \dots &0 & 0& 0 & t(p-1)\\
    0 &\dots &0 & 0& t(p-1) & 0
\end{bmatrix}.
\] 

The diagonalization of this Gram matrix of the diagonal bilinear form $\left(\frac{p-1}{2}\right)\cdot h+\langle (p-1)\rangle=\left(\frac{p-3}{2}\right)\cdot h+\langle -1\rangle$. Since $\deg_{\infty}^\Aone \frac{(1-p)v^{p-2}}{-v^pt-1} = \langle 1 \rangle + \langle -1 \rangle$, it follows that $\ind_{\infty}df=\deg_0^\Aone \frac{(1-p)v^{p-2}}{-v^pt-1}=\left(\frac{p-3}{2}\right)\cdot h+\langle - 1\rangle$ in GW($k(t)$). 

We conclude that 
\begin{align*} 
\sum_{\{y\colon df(y)=0\}}\ind_y^\Aone df&= \ind_{(y^p-t)}^\Aone df+\ind_{\infty}^\Aone df\\
&=\left(\frac{p-1}{2}\right)h+\langle 1\rangle +h\left(\frac{p-3}{2}\right)+\langle -1\rangle \\
&=(p-2)h+\langle 1\rangle +\langle -1\rangle\\
&=(p-1)h,\\
\end{align*}

as desired. \end{proof}

\begin{pr}\label{ExampleTwo} The rational function  $f: \bbP^1_{k(t),y} \to \bbP^1_{k(t),x}$ defined by $y\mapsto y^{p}-y$ has the property that there is a relative orientation such that the branch indices are
\begin{align*}
	ind_{\infty}df=& \left(\frac{p-1}{2}\right)\cdot h.
\end{align*}

In particular, the enriched Riemann--Hurwitz formula \eqref{Eqn: FinalRH} holds.
\end{pr}
\begin{proof}
Orient as in the previous proof.  For the affine coordinates $z$ and $v$ as in that proof, we have that $f$ is given by $v \mapsto v^{p}/(1-v^{p-1})$, so 
\[
	df|U = \frac{ -v^{2 p-2}}{(1-v^{p-1})^2} \cdot dv.
\]
We complete the proof by computing the local degree of $\frac{ -v^{2 p-2}}{(1-v^{p-1})^2}$ using \cite{cazanavea}.  An algebra computation shows that the 
\begin{multline*}
	\frac{(-x^{2 p-2}) (1-y^{p-1})^2-(-y^{2 p-2}) (1-x^{p-1})^2}{x-y} = -(x^{2 p-3} + x^{2 p-4} y + \dots + x y^{2 p-4} + y^{2 p-3}) \\+ 2 x^{p-1} y^{p-1} (x^{p-2}+x^{p-3} y + \dots + x y^{p-3} y^{p-2}).
\end{multline*}
We conclude that a Gram matrix for the local degree is
\[
\begin{bmatrix}
 0 & 0 & \dots & 0 & 0	& \dots & 0 & -1 \\
 0 & 0 & \dots & 0 & 0 	& \dots & -1 & 0 \\
 \vdots & \vdots & \ddots & \vdots & \vdots & \vdots & \vdots & \vdots \\
  0 & 0 & \dots & 0 & -1 	& \dots & 0 & 2 \\
    0 & 0 & \dots & -1 & 0 	& \dots & 2 & 0 \\
  \vdots & \vdots & \ddots & \vdots & \vdots & \vdots & \vdots & \vdots \\
 0 & -1 & \dots & 0 & 2 	& \dots & 0 & 0 \\
 -1 & 0 & \dots & 2 & 0 	& \dots & 0 & 0 \\
\end{bmatrix}.
\] 
This Gram matrix is equivalent to the Gram matrix where the only nonzero entries are $-1$'s along the antidiagonal, and this matrix represents $\left(\frac{p-1}{2}\right)\cdot h$.
\end{proof}

\section{Main Theorem}\label{Section:MT}

We prove the enriched Riemann--Hurwitz formula, Theorem~\ref{IntroRHthm}, over an arbitrary field $k$ discussed in the introduction. Marc Levine has previously shown an enriched Riemann--Hurwitz formula \cite[Theorem 12.7]{Levine-EC}. The hypotheses and context of Theorem~\ref{IntroRHthm} differ from M. Levine's result, and our particular interest in the present context comes from the possibility of explicitly computing certain local indices, even in the presence of wild ramification and inseparable residue field extensions. 

Let $f: Y \to X$ be a non-constant, separable map of smooth, proper, geometrically connected curves over $k$.

We will need Nisnevich coordinates around the closed points of $Y$ (see Section \ref{Notationsection} for the definition of Nisnevich coordinates). For $k$ infinite, the existence of such coordinates is shown by \cite[Chapter 8, Proposition 3.2.1]{knus}. When $k \subseteq k(y)$ is separable, this is proven in \cite[Lemma 18]{CubicSurface}. Combining these results gives the desired existence. We include a different proof here, directly extending \cite[Lemma 18]{CubicSurface} in the case where the dimension of $Y$ is $1$, because it offers another perspective. Namely, the proof of \cite[Lemma 18]{CubicSurface} holds under the weaker hypothesis that $k \subseteq k(y)$ is a simple extension of fields, meaning that $k(y)$ is obtained from $k$ by adjoining a single element. We will show that $k \subseteq k(y)$ is always simple, using a modification of David Speyer's proof of the Primitive Element Theorem \cite{SpeyerPrimElement}. 

\begin{lm}[Speyer, Lemma 2, loc. cit]\label{SpeyerLemma2}
Let $r(x)$ and $q(x)$ be polynomials with coefficients in a field with $q(0) \neq 0$. Then, for all but finitely many $t$, the polynomials $r(tx)$ and $q(x)$ have no common factor.
\end{lm}

Speyer's proof is cleverly designed to avoid field extensions; we include a prosaic one for convenience.

\begin{proof}
In an algebraic closure, we may factor $r(x)=a(x-\alpha_1)(x-\alpha_2)\cdots(x-\alpha_n)$ and $q(x)=b(x-\beta_1)(x-\beta_2)\cdots(x-\beta_m)$. If $r(tx)$ and $q(x)$ have a common factor, then there is some $\alpha_i$ and $\beta_j$ such that $tx - \alpha_i$ is a multiple of $x-\beta_j$, or equivalently $\alpha_i = t \beta_j$. As $\beta_j$ is not zero, this eliminates only one value of $t$.
\end{proof}

\begin{lm}\label{separable_extension_of_simple_extension_is_simple}
Let $k \subseteq M$ be a finite simple field extension. Let $M \subseteq E$ be a finite separable field extension. Then $k \subseteq E$ is a simple extension.
\end{lm}

\begin{proof}
We may assume that $k$ is an infinite field of characteristic $p >0$, because the result is immediate when $k$ is finite or characteristic $0$. 

By assumption, there is an element $\alpha$ of $M$ such that $M = k[\alpha]$. Let $f$ be the minimal polynomial of $\alpha$ over $k$. 

Since $M \subseteq E$ is finite and separable, there exists $\beta$ in $E$ such that $E = M[\beta]= k[\alpha, \beta]$. Since $\beta$ is separable over $M$, $E = M[\beta^{p^d}]$ for all $d$ by \cite[Chapter~V, Exercise~16, page~254]{LangAlgebra}. There exists a $d$ such that $\beta^{p^d}$ is separable over $k$\hidden{Let $g$ be the minimal polynomial of $\beta$ over $k$. Then take $d$ maximal such that $g (x) = g_1(x^{p^d})$. $g_1$ is irreducible with nonzero derivative and therefore separable.}. Thus, by replacing $\beta$ by $\beta^{p^d}$, we may assume that $\beta$ is separable over $k$. Let $g$ be the minimal polynomial of $\beta$ over $k$.  

Define $r(x)$ and $q(x)$ in $E[x]$ by $f(x) = (x-\alpha)r(x-\alpha)$ and $g(x) = (x-\beta)q(x-\beta)$. Since $\beta$ is separable, we know that $0$ is not a root of $q(x)$. Therefore by Lemma \ref{SpeyerLemma2} and our assumption that $k$ is infinite, we can choose $t$ in $k$ such that $r(tx)$ and $q(x)$ have no common factor. We claim that $E =k[\alpha - t \beta]$.

To see this, let $h(x) = f(tx + \alpha - t \beta)$. Note that $h(x)$ has coefficients in $k[\alpha - t \beta]$. Since the polynomial ring $k[\alpha - t \beta][x]$ is a principal ideal domain, the ideal $\langle h(x), g(x) \rangle$ is generated by a polynomial $s(x)$ in $k[\alpha - t \beta][x]$, which we may assume to be monic. (So $s(x)$ is the GCD, but we wish to emphasize the ambient ring.)The polynomial $s(x)$ also generates the ideal $\langle h(x), g(x) \rangle$  of $E[x]$ generated $h(x)$ and $g(x)$. 

We compute $s(x)$ by computing the GCD of $h(x)$ and $g(x)$ in $E[x]$. We have $h(x) = f(tx + \alpha - t \beta) = t(x - \beta) r(t(x - \beta))$ and $g(x) =(x-\beta)q(x-\beta).$ By the choice to $t$, the polynomials $ r(t(x - \beta))$ and $q(x-\beta)$ have no common factor. Therefore, the GCD of $h(x)$ and $g(x)$ is $x-\beta$. Greatest common divisors are well-defined up to a nonzero scalar. Since $x-\beta$ and $s(x)$ are both GCDs of $h(x)$ and $g(x)$ and are both monic polynomials, we have that $s(x) = x-\beta$. Thus $\beta$ is in $k[\alpha - t \beta]$. Since $t$ is in $k$, it follows that $\alpha$ is in $k[\alpha - t \beta]$. Thus $k[\alpha - t \beta]\supset k[\alpha,\beta] = E$, proving the claim. 
\end{proof}

We combine the previous lemma to reprove the existence of the desired Nisnevich coordinates:

\begin{pr}\label{NisnevichCoord}
Let $Y$ be a smooth curve over $k$ and let $y$ be a closed point of $Y$. Then there exist Nisnevich coordinates around $y$.
\end{pr}

\begin{rmk}
Proposition~\ref{NisnevichCoord} follows from \cite[Chapter 8, Proposition 3.2.1]{knus} when $k$ is infinite.
\end{rmk}

\begin{proof}
By the proof of \cite[Lemma 18]{CubicSurface}, Nisnevich coordinates exist around $y$ when $k \subseteq k(y)$ is a finite, simple extension of fields. Since $Y$ is smooth, there is an \'etale map $\phi$ from a Zariski open neighborhood of $y$ to $\mathbb{A}^1_k$. It follows that $k \subseteq k(y)$ is of the form $k \subseteq k(\phi(y)) \subseteq k(y)$ where $k \subseteq k(\phi(y))$ is a simple extension and $ k(\phi(y)) \subseteq k(y)$  is a separable extension. The proposition thus follows from Lemma~\ref{separable_extension_of_simple_extension_is_simple}.
\end{proof}

To prove Theorem~\ref{IntroRHthm}, we prove that under the same hypotheses, there is an equality $$\sum_{\{y: df(y) = 0\}} \ind_y df = h(g(Y) -1 + \deg f(1-g(X)))$$ in $\GW(k)$ between the local indices (or degrees) of $df$ at its zeros and the hyperbolic form $h=\langle -1 \rangle + \langle 1 \rangle$ multiplied by the integer $g(Y) -1 + \deg f(1-g(X))$, where $g(X)$ and $g(Y)$ denote the genera of $X$ and $Y$, respectively, and $\deg f = [k(Y):k(X)]$.

\begin{proof}[Proof of Theorem~\ref{IntroRHthm}]
Since $f$ is separable, the section $df$ of $$\mathcal{V}= \Hom(f^* T^{\vee}X ,T^{\vee}Y)$$ is nonzero and therefore has only isolated zeros, because $Y$ is a curve. By Proposition~\ref{NisnevichCoord}, there are Nisnevich coordinates around all the zeros of $df$.

We claim that $\mathcal{V}$ is relatively orientable. By Assumption \eqref{assumeTXsquare}, we may choose a line bundle $M$ on $X$ such that $M^{\otimes 2}\cong T^{\vee}X$. Then 
\begin{align*}
\Hom(TY, \mathcal{V})&\cong T^{\vee}Y\otimes\Hom(f^{\vee}T^{\vee}X,T^{\vee}Y)\\
&\cong T^{\vee}Y\otimes(f^{*}T^{\vee}X)^{\vee}\otimes T^{\vee}Y\\
&\cong (T^{\vee}Y)^{\otimes 2}\otimes ((f^{*}M)^{\vee})^{\otimes 2} \\
&\cong (T^{\vee}Y\otimes(f^{\vee}M)^{\vee})^{\otimes 2},\\
\end{align*}
and so $\Hom(TY,\mathcal{V})$ is a square. We may therefore choose a relative orientation, and we do this now. (It will not matter what the chosen relative orientation is in the present case.)

Therefore, $e(Y, \mathcal{V}, df)$ is defined, and by definition is equal to \begin{equation}\label{eissumlocal}e(Y, \mathcal{V}, df) =  \sum_{y: df(y) = 0} \ind_y df.\end{equation}

Since $\mathcal{V}= \Hom(f^* T^{\vee}X ,T^{\vee}Y)$ is rank one, which is odd, $e(Y, \mathcal{V}, df)$ is a multiple of the hyperbolic element $h$ by \cite[Proposition 12]{FourLines} and Assumption \ref{technical}. (In a different context and under different hypotheses \cite[Proposition 12]{FourLines} is proven by M. Levine in \cite[Theorem 7.1]{Levine-EC}. M. Levine also credits J. Fasel.) Thus

\begin{equation}\label{eismulth} e(Y,\mathcal{V},df) =h(\deg\mathcal{V}/2). \end{equation}  Using the fact that $\Hom(f^* T^{\vee}X ,T^{\vee}Y)=(f^*T^{\vee}X)^{\vee}\otimes T^{\vee}Y$, we have

  \begin{align}\label{degV}\deg \Hom(f^* T^{\vee}X ,T^{\vee}Y) &= \deg(f^*T^{\vee}X)^{\vee}\otimes T^{\vee}Y\\&=\deg f^* TX  + \deg T^{\vee}Y \\ &=  \deg f \deg TX  + \deg T^{\vee}Y  
\\ &=  \deg f (2-2 g(X))  + (2 g(Y) -2).\end{align} Combining Equations \eqref{eissumlocal}, \eqref{eismulth}, and \eqref{degV} gives the desired result.
\end{proof}

\subsection{Acknowledgements}
We gratefully thank Shuji Saito for raising the question at the workshop {\em Motivic homotopy theory and refined enumerative geometry} in Essen, and for interesting discussions about it. We likewise thank Alexey Ananyevskiy and Ivan Panin for the reference to \cite[Chapter 8, Proposition 3.2.1]{knus} for the existence of Nisnevich coordinates.

Candace Bethea was partially supported by a SPARC Graduate Research Grant from the Office of the Vice President for Research at the University of South Carolina.

Jesse Kass was partially supported by the Simons Foundation under Award Number 429929.
 
Kirsten Wickelgren was partially supported by National Science Foundation Award DMS-1552730.

\bibliographystyle{amsalpha}

%
\bibliography{SaitoQuestion}

\begin{bibdiv}
\begin{biblist}

\bib{AsokFasel_comp_euler_classes}{article}{
      author={Asok, A.},
      author={Fasel, J.},
       title={Comparing {E}uler classes},
        date={2016},
        ISSN={0033-5606},
     journal={Q. J. Math.},
      volume={67},
      number={4},
       pages={603\ndash 635},
}

\bib{abbes02}{article}{
      author={Abbes, A,},
      author={Saito, T.},
       title={Ramification of local fields with imperfect residue fields},
        date={2002},
        ISSN={0002-9327},
     journal={Amer. J. Math.},
      volume={124},
      number={5},
       pages={879\ndash 920},
  url={http://muse.jhu.edu/journals/american_journal_of_mathematics/v124/124.5abbes.pdf},
}

\bib{abbes03}{article}{
      author={Abbes, A.},
      author={Saito, T.},
       title={Ramification of local fields with imperfect residue fields.
  {II}},
        date={2003},
        ISSN={1431-0635},
     journal={Doc. Math.},
      number={Extra Vol.},
       pages={5\ndash 72},
        note={Kazuya Kato's fiftieth birthday},
}

\bib{BargeMorel}{article}{
      author={Barge, J.},
      author={Morel, F.},
       title={Groupe de {C}how des cycles orient\'es et classe d'{E}uler des
  fibr\'es vectoriels},
        date={2000},
        ISSN={0764-4442},
     journal={C. R. Acad. Sci. Paris S\'er. I Math.},
      volume={330},
      number={4},
       pages={287\ndash 290},
  url={http://dx.doi.org.prx.library.gatech.edu/10.1016/S0764-4442(00)00158-0},
}

\bib{Bachmann}{unpublished}{
      author={Bachmann, T.},
      author={Wickelgren, K.},
       title={Thirteen ways of looking at an {A}1-{E}uler class},
        date={2019},
        note={In preparation},
}

\bib{cazanaveb}{article}{
      author={Cazanave, C.},
       title={Classes d'homotopie de fractions rationnelles},
        date={2008},
        ISSN={1631-073X},
     journal={C.~R.~Math.~Acad.~Sci.~Paris},
      volume={346},
      number={3-4},
       pages={129\ndash 133},
         url={http://dx.doi.org/10.1016/j.crma.2008.01.004},
}

\bib{cazanavea}{article}{
      author={Cazanave, C.},
       title={Algebraic homotopy classes of rational functions},
        date={2012},
        ISSN={0012-9593},
     journal={Ann.~Sci.~\'Ec.~Norm.~Sup\'er. (4)},
      volume={45},
      number={4},
       pages={511\ndash 534 (2013)},
}

\bib{DJK}{unpublished}{
      author={D\'eglise, F.},
      author={Jin, F.},
      author={Khan, A.},
       title={Fundamental classes in motivic homotopy theory},
        date={2018},
        note={{\em Preprint}, available at
  \url{https://arxiv.org/abs/1805.05920}},
}

\bib{FaselGroupesCW}{article}{
      author={Fasel, J.},
       title={Groupes de {C}how-{W}itt},
        date={2008},
        ISSN={0249-633X},
     journal={M\'em. Soc. Math. Fr. (N.S.)},
      number={113},
       pages={viii+197},
}

\bib{Grig_Ivan}{article}{
      author={Grigor`ev, D.~Ju.},
      author={Ivanov, N.~V.},
       title={On the {E}isenbud-{L}evine formula over a perfect field},
        date={1980},
        ISSN={0002-3264},
     journal={Dokl. Akad. Nauk SSSR},
      volume={252},
      number={1},
       pages={24\ndash 27},
}

\bib{hartshorne}{book}{
      author={Hartshorne, R.},
       title={Algebraic geometry},
   publisher={Springer-Verlag, New York-Heidelberg},
        date={1977},
        ISBN={0-387-90244-9},
        note={Graduate Texts in Mathematics, No. 52},
}

\bib{kato89}{incollection}{
      author={Kato, K.},
       title={Swan conductors for characters of degree one in the imperfect
  residue field case},
        date={1989},
   booktitle={Algebraic {$K$}-theory and algebraic number theory ({H}onolulu,
  {HI}, 1987)},
      series={Contemp. Math.},
      volume={83},
   publisher={Amer. Math. Soc., Providence, RI},
       pages={101\ndash 131},
         url={https://doi.org/10.1090/conm/083/991978},
}

\bib{knus}{book}{
      author={Knus, M.-A.},
       title={Quadratic and {H}ermitian forms over rings},
      series={Grundlehren der Mathematischen Wissenschaften [Fundamental
  Principles of Mathematical Sciences]},
   publisher={Springer-Verlag, Berlin},
        date={1991},
      volume={294},
        ISBN={3-540-52117-8},
         url={http://dx.doi.org/10.1007/978-3-642-75401-2},
        note={With a foreword by I. Bertuccioni},
}

\bib{CubicSurface}{unpublished}{
      author={Kass, J.~L.},
      author={Wickelgren, K.},
       title={An arithmetic count of the lines on a smooth cubic surface},
        date={2017},
        note={{\em Preprint}, available at arXiv:1708.01175},
}

\bib{KWA2degree}{unpublished}{
      author={Kass, J.~L.},
      author={Wickelgren, K.},
       title={A classical proof that the algebraic homotopy class of a rational
  function is the residue pairing},
        date={2018},
        note={{\em Preprint}, available at arXiv:1602.08129v2},
}

\bib{KWA1degree}{article}{
      author={Kass, J.~L.},
      author={Wickelgren, Kirsten},
       title={The class of {E}isenbud-{K}himshiashvili-{L}evine is the local
  {$\bold{A}^1$}-{B}rouwer degree},
        date={2019},
        ISSN={0012-7094},
     journal={Duke Math. J.},
      volume={168},
      number={3},
       pages={429\ndash 469},
         url={https://doi.org/10.1215/00127094-2018-0046},
}

\bib{LangAlgebra}{book}{
      author={Lang, S.},
       title={Algebra},
     edition={Third Edition},
      series={Graduate Texts in Mathematics},
   publisher={Springer-Verlag, New York},
        date={2002},
      volume={211},
        ISBN={0-387-95385-X},
         url={https://doi.org/10.1007/978-1-4613-0041-0},
}

\bib{Levine-EC}{unpublished}{
      author={Levine, M.},
       title={Toward an enumerative geometry with quadratic forms},
        date={2017},
        note={{\em Preprint}, available at
  \url{https://arxiv.org/abs/1703.03049}},
}

\bib{LevineRaksit_MotivicGaussBonnet}{unpublished}{
      author={Levine, M.},
      author={Raksit, A.},
       title={Motivic {G}auss--{B}onnet formulas},
        date={2018},
        note={{\em Preprint}, available at
  \url{https://arxiv.org/abs/1808.08385}},
}

\bib{qing}{book}{
      author={Liu, Q.},
       title={Algebraic geometry and arithmetic curves},
      series={Oxford Graduate Texts in Mathematics},
   publisher={Oxford University Press, Oxford},
        date={2002},
      volume={6},
        ISBN={0-19-850284-2},
        note={Translated from the French by R.~Ern\'{e}, Oxford Science
  Publications},
}



\bib{morel}{book}{
      author={Morel, F.},
       title={{$\Bbb A^1$}-algebraic topology over a field},
      series={Lecture Notes in Mathematics},
   publisher={Springer, Heidelberg},
        date={2012},
      volume={2052},
        ISBN={978-3-642-29513-3},
         url={http://dx.doi.org/10.1007/978-3-642-29514-0},
}

\bib{neukirch}{book}{
      author={Neukirch, J.},
       title={Algebraic number theory},
      series={Grundlehren der Mathematischen Wissenschaften [Fundamental
  Principles of Mathematical Sciences]},
   publisher={Springer-Verlag, Berlin},
        date={1999},
      volume={322},
        ISBN={3-540-65399-6},
         url={https://doi.org/10.1007/978-3-662-03983-0},
        note={Translated from the 1992 German original and with a note by
  N.~Schappacher, With a foreword by G. Harder},
}

\bib{okonek14}{article}{
      author={Okonek, C.},
      author={Teleman, A.},
       title={Intrinsic signs and lower bounds in real algebraic geometry},
        date={2014},
        ISSN={0075-4102},
     journal={J. Reine Angew. Math.},
      volume={688},
       pages={219\ndash 241},
         url={http://dx.doi.org/10.1515/crelle-2012-0055},
}

\bib{saito12}{article}{
      author={Saito, T.},
       title={Ramification of local fields with imperfect residue fields
  {III}},
        date={2012},
        ISSN={0025-5831},
     journal={Math. Ann.},
      volume={352},
      number={3},
       pages={567\ndash 580},
         url={https://doi.org/10.1007/s00208-011-0652-5},
}

\bib{serre79}{book}{
      author={Serre, J.-P.},
       title={Local fields},
      series={Graduate Texts in Mathematics},
   publisher={Springer-Verlag, New York-Berlin},
        date={1979},
      volume={67},
        ISBN={0-387-90424-7},
        note={Translated from the French by M.~Greenberg},
}

\bib{SpeyerPrimElement}{misc}{
      author={Speyer, D.},
       title={Math {O}verflow: \itshape primitive element theorem without
  building field extensions},
  how={\url{https://mathoverflow.net/questions/29687/primitive-element-theorem-without-building-field-extensions}},
        date={2010},
}

\bib{scheja}{article}{
      author={Scheja, G.},
      author={Storch, U.},
       title={\"{U}ber {S}purfunktionen bei vollst\"andigen {D}urchschnitten},
        date={1975},
        ISSN={0075-4102},
     journal={J. Reine Angew. Math.},
      volume={278/279},
       pages={174\ndash 190},
}

\bib{FourLines}{unpublished}{
      author={Srinivasan, P.},
      author={Wickelgren, K.},
       title={An arithmetic count of the lines meeting four lines in
  $\mathbb{P}^3$},
        date={2018},
        note={{\em Preprint}, available at
  \url{https://arxiv.org/abs/1810.03503}},
}

\end{biblist}
\end{bibdiv}

\end{document}